\documentclass{amsart}
\usepackage{amssymb}
\usepackage{hyperref}
\DeclareMathOperator{\re}{Re}
\DeclareMathOperator{\im}{Im}
\DeclareMathOperator{\supp}{supp}
\newtheorem{theorem}{Theorem}[section]
\newtheorem{lemma}[theorem]{Lemma}
\newtheorem{proposition}[theorem]{Proposition}
\newtheorem{corollary}[theorem]{Corollary}
\numberwithin{equation}{section}

\title[Fourier-analytic approach to vertical projections]{A Euclidean Fourier-analytic approach to vertical projections in the Heisenberg group}
\author{Terence L.~J.~Harris}
\address{Department of Mathematics, Cornell University, Ithaca, NY 14853, USA}
\email{tlh236@cornell.edu}
\subjclass[2020]{28A78; 28A80}
\keywords{Heisenberg group, Hausdorff dimension, vertical projections}

\begin{document} 
\begin{abstract} An improved a.e.~lower bound is given for Hausdorff dimension under vertical projections in the first Heisenberg group, with respect to the Carnot-Carathéodory metric. This improves the known lower bound, and answers a question of Fässler and Hovila. The approach uses the Euclidean Fourier transform, Basset's integral formula, and modified Bessel functions of the second kind.  \end{abstract}
\maketitle

\section{Introduction} Let $\mathbb{H}$ be the first Heisenberg group, identified with $\mathbb{C} \times \mathbb{R}$ and equipped with the group law
\[ (z,t) \ast (\zeta, \tau) = (z+\zeta, t+\tau + 2 \omega(z, \zeta) ), \]
where $\omega(z,\zeta) = \im \left( z \overline{\zeta} \right)$.  The Carnot-Carathéodory metric on $\mathbb{H}$ is bi-Lipschitz equivalent to the Korányi metric
\[ d_{\mathbb{H}}((z,t), (\zeta, \tau) ) = \left\lVert (\zeta, \tau)^{-1}  \ast (z,t) \right\rVert_{\mathbb{H}}, \]
where 
\[ \left\lVert (z,t)\right\rVert_{\mathbb{H}}  = \left( \left\lvert z\right\rvert^4 + t^2 \right)^{1/4}; \]
see \cite[pp.~18--19]{capogna}. This work gives an improved a.e.~lower bound for the Hausdorff dimension of sets under vertical projections in $\mathbb{H}$, where the Hausdorff dimension $\dim A$ of a set $A \subseteq \mathbb{H}$ is defined through the Korányi metric (equivalently the Carnot-Carathéodory metric). The definition of the vertical projections will be summarised briefly here, but see \cite{balogh} and \cite{balogh2} for more background. 

For each $\theta \in [0, \pi)$, let 
\[ \mathbb{V}_{\theta}^{\perp} = \left\{ \left(\lambda_1 ie^{i \theta}, \lambda_2 \right) \in \mathbb{C} \times \mathbb{R} : \lambda_1, \lambda_2 \in \mathbb{R} \right\}, \]
and 
\[ \mathbb{V}_{\theta} = \left\{ \left(\lambda e^{i \theta}, 0\right) \in \mathbb{C} \times \mathbb{R} : \lambda \in \mathbb{R} \right\}. \]
Then each $(z,t) \in \mathbb{H}$ can be uniquely written as a product 
\[ (z,t) = P_{\mathbb{V}_{\theta}^{\perp}}(z,t) \ast P_{\mathbb{V}_{\theta}}(z,t) \]
 of an element of $\mathbb{V}_{\theta}^{\perp}$ on the left, with an element of $\mathbb{V}_{\theta}$ on the right. For each $\theta \in [0, \pi)$, this defines the vertical projection $P_{\mathbb{V}_{\theta}^{\perp}}$ and the horizontal projection $P_{\mathbb{V}_{\theta}}$. A formula for $P_{\mathbb{V}_{\theta}^{\perp}}$ is 
\[ P_{\mathbb{V}_{\theta}^{\perp}}(z,t) = \left( \pi_{V_{\theta}^{\perp}}(z), t + 2\omega\left( \pi_{V_{\theta}}(z), z \right) \right), \]
where $\pi_{V_{\theta}^{\perp}}$ is the orthogonal projection onto the line in $\mathbb{R}^2$ with direction $ie^{i \theta}$, and $\pi_{V_{\theta}}$ is the orthogonal projection onto the line in $\mathbb{R}^2$ with direction $e^{i \theta}$. 

In \cite[Conjecture~1.5]{balogh} it was conjectured that for any (presumably Borel or analytic) set $A \subseteq \mathbb{H}$, $\dim P_{\mathbb{V}_{\theta}^{\perp}}(A) \geq \min\{ \dim A, 3\}$ for a.e.~$\theta \in [0, \pi)$, and that if $\dim A > 3$ then $P_{\mathbb{V}_{\theta}^{\perp}}(A)$ has positive area for a.e.~$\theta \in [0, \pi)$. This conjecture is known in the range $\dim A \leq 1$; see \cite[Theorem~1.4]{balogh}.  In \cite{fassler} and \cite{me}, some improvements were made beyond the lower bound $\dim P_{\mathbb{V}_{\theta}^{\perp}}(A) \geq 1$ for sets with $\dim A > 2$. Question~4.2 from \cite{fassler} asked whether any improvement over the lower bound of 1 was possible for sets of dimension between 1 and 2. The following theorem gives a positive answer. 

\begin{theorem} \label{mainthm} Let $A \subseteq \mathbb{H}$ be an analytic set with $\dim A > 1$. Then 
\[ \dim P_{\mathbb{V}_{\theta}^{\perp}}(A) \geq \min\left\{ \frac{ 1+\dim A }{2}, 2 \right\}, \]
for a.e.~$ \theta \in [0, \pi)$. 
\end{theorem}
This improves\footnote{After this paper was written, a result of Fässler and Orponen appeared which is better than Theorem~\ref{mainthm} for $\dim A > 7/3$ \cite{fasslerorponen}.} the known lower bound in the range $1< \dim A < 7/2$. If $\dim A \geq 7/2$, then the lower bound $\dim P_{\mathbb{V}_{\theta}^{\perp}}(A) \geq 2\dim A - 5$ from \cite{balogh} is better than Theorem~\ref{mainthm} and holds for every $\theta \in [0, \pi)$. A special case of the lower bound in Theorem~\ref{mainthm} was proved in \cite[Theorem~7.10]{balogh} for sets contained in a vertical subgroup. 

The proof of Theorem~\ref{mainthm} uses the Euclidean Fourier transform. An approach to Hausdorff dimension via the (non-Euclidean) group Fourier transform was developed by Rom\'{a}n-Garc\'{i}a \cite{romangarcia}, who proved a group Fourier-analytic formula for the energy of a measure, via the group Fourier transform of the Korányi kernels $\left\lVert \cdot \right\rVert_{\mathbb{H}}^{-s}$. Unlike the group Fourier transform case, the Euclidean Fourier transforms of the Korányi kernels seem to be unexplored. In Lemma~\ref{koranyipotential}, it is shown that if $s \in (1,3)$, then
\[ 0 <  \widehat{f_s} \lesssim f_{3-s}, \]
where $f_s(x,t) = (x^4 + t^2)^{-s/4}$ for $(x,t) \in \mathbb{R}^2$. This seems to be a partial analogue of the formula $\widehat{k_s} = c_{n,s}k_{n-s}$ for the Riesz kernels $k_s(x) = \left\lvert x\right\rvert^{-s}$ on $\mathbb{R}^n$. Only the case $s \in (1,2)$ of the inequality $\widehat{f_s} \lesssim f_{3-s}$ is used in the proof of Theorem~\ref{mainthm}.

The (Euclidean) Fourier transform of a compactly supported finite Borel measure $\mu$ on $\mathbb{R}^n$ is a locally Lipschitz function, defined by 
\[ \widehat{\mu}(\xi) = \int e^{-2\pi i \langle x, \xi \rangle } \, d\mu(x), \qquad \xi \in \mathbb{R}^n. \]
Frostman's lemma will be used throughout, which states that if $A$ is an analytic (Suslin) subset of a complete separable metric space $(X,d)$, then 
\begin{align*}  \dim A &= \sup\{ s \geq 0: \exists \mu \in \mathcal{M}(A) \text{ with } I_s(\mu) < \infty \} \\
&=  \sup\{ s \geq 0: \exists \mu \in \mathcal{M}(A) \text{ with } c_s(\mu) < \infty \}, \end{align*}
where
\[ I_s(\mu) = \int \int d(x,y)^{-s} \, d\mu(x) \, d\mu(y), \]
\[ c_s(\mu) = \sup_{\substack{ x \in X \\ r >0} } \frac{ \mu (B(x,r) )}{r^s}, \]
and $\mathcal{M}(A)$ is the set of nonzero finite Borel measures compactly supported on $A$. In particular, Frostman's lemma holds for Borel sets, since Borel sets are analytic. For an introduction to energies $I_s(\mu)$ see \cite[Chapter~8]{mattila2}, and see e.g.~\cite[Appendix B]{bishop} for a proof of Frostman's lemma in the general case. See~\cite{kaufman} for the first application of energies to projections. 

Section~\ref{bessel} contains most of the background on modified Bessel functions needed in Section~\ref{mainsec}, and Section~\ref{mainsec} contains the proofs of the main results. Section~\ref{remark} has some remarks and further questions. 

\section{Background on modified Bessel functions} 
\label{bessel}

Define the modified Bessel function of the first kind, of order $\nu$, by
\begin{equation} \label{besseldef} I_{\nu}(z) = \sum_{n=0}^{\infty} \frac{ (z/2)^{\nu + 2n } }{n! \Gamma(n+\nu +1) }, \end{equation}
for all $z \in \mathbb{C} \setminus \{0\}$ when $\nu \in \mathbb{C} \setminus \mathbb{Z}$, and for all $z \in \mathbb{C}$ when $\nu \in \mathbb{Z}$. To make $I_{\nu}$ a single-valued function, the function $z^{\nu}$ is defined to be $e^{ \nu \log z}$ where $-\pi < \arg z \leq \pi$, unless mentioned otherwise. By convention the sum in \eqref{besseldef} starts at $-\nu$ when $\nu$ is a negative integer. Then $I_{\nu}$ is an entire function when $\nu \in \mathbb{Z}$, and is analytic on $\mathbb{C} \setminus (-\infty, 0]$ when $\nu \in \mathbb{C} \setminus \mathbb{Z}$. 

Define the modified Bessel function of the second kind, of order $\nu$, by 
\[  K_{\nu}(z) = \frac{\pi}{2 \sin (\nu \pi ) } \left( I_{-\nu}(z) - I_{\nu}(z) \right), \]
when $\nu \in \mathbb{C} \setminus \mathbb{Z}$, and for any $n \in \mathbb{Z}$, define
\begin{equation} \label{integerK} K_{n}(z) = \lim_{\nu \to n} K_{\nu}(z).  \end{equation}
For any $\nu \in \mathbb{C}$, the domain of $K_{\nu}$ is $\mathbb{C} \setminus \{0\}$, and $K_{\nu}$ is analytic on $\mathbb{C} \setminus (-\infty, 0]$. For any fixed $z \in \mathbb{C} \setminus \{0\}$, the limit in \eqref{integerK} exists since $I_{\nu}(z)$ is an entire function of $\nu$, and so the limit in \eqref{integerK} can be expressed as a difference of partial derivatives with respect to $\nu$. For fixed nonzero $z$, the function $K_{\nu}(z)$ is continuous at $\nu = n$, for any $n \in \mathbb{Z}$. Finally, the definition implies that $K_{\nu} = K_{-\nu}$ for all $\nu \in \mathbb{C}$. 
\begin{proposition}[{\cite[p.~79]{watson}}] \label{derivativeidentity} For any $\nu \in \mathbb{C}$ and $z \in \mathbb{C} \setminus (-\infty, 0 ]$,
\[ \frac{d}{dz} \left[ z^{\nu} K_{\nu}(z) \right] = -z^{\nu} K_{\nu-1}(z), \]
and
\[ \frac{d}{dz} \left[ z^{-\nu} K_{\nu}(z) \right] = -z^{-\nu} K_{\nu + 1}(z). \] \end{proposition} 

The third formula for $K_{\nu}$ in the theorem below is known as Basset's integral formula, and a proof of the theorem can be found in \cite[p.~172]{watson}. A more direct proof is outlined in \cite[p.~384]{watsonwhittaker}, though the definition of $K_{\nu}$ given in \cite{watsonwhittaker} has an extra factor of $\cos( \pi \nu)$ compared to the (now) standard definition. 
\begin{theorem}[{\cite[p.~172]{watson}}] \label{basset} If $\re \nu > -1/2$ and $\re z > 0$, then 
\begin{align*} K_{\nu}(z) &= \frac{(z/2)^{\nu} \sqrt{\pi} }{ \Gamma\left( {\nu}+\frac{1}{2} \right)} \int_0^{\infty} e^{-z \cosh \phi} (\sinh \phi)^{2{\nu}} \, d\phi \\
&= \frac{(z/2)^{\nu}\sqrt{\pi}   }{ \Gamma\left( {\nu}+\frac{1}{2} \right)} \int_1^{\infty} e^{-zt} (t^2-1)^{{\nu}- \frac{1}{2} } \, dt \\
&= \frac{  (2z)^{\nu} \Gamma\left( {\nu} + \frac{1}{2} \right) }{2\sqrt{\pi}} \int_{-\infty}^{\infty} e^{-i u } (z^2+u^2)^{-{\nu} - \frac{1}{2}} \, du, \end{align*}
where the last integral is an improper Riemann integral. In particular (due to either of the first two formulas), $K_{\nu}(x)$ is strictly positive for all $x>0$ and $\nu \in \mathbb{R}$.
  \end{theorem}
One corollary of the preceding theorem is the (known) limit for $K_{\nu}$ stated in Corollary~\ref{asymptotics} below. This limit follows from taking the first term in the more general asymptotic series expansion for $K_{\nu}$ (see e.g.~\cite[p.~202]{watson}), but the limit can also be calculated easily from Theorem~\ref{basset}.
\begin{corollary} \label{asymptotics} For any $\nu \in \mathbb{C}$,
\[ \lim_{x \to +\infty} \frac{K_{\nu}(x)}{\left( \frac{\pi}{2x} \right)^{1/2}e^{-x}} = 1.  \]
 \end{corollary}	
Another (known) corollary of Theorem~\ref{basset} is that the singularity of $K_0$ at the origin is of logarithmic type, which follows from the series for $K_0$ (see~\cite[p.~80]{watson}), but a short proof is included below. 
\begin{corollary} \label{besselzero} For all $x>0$,
\[ K_0(x) \leq 1+ \log 2 +  \left\lvert \log x\right\rvert. \]  \end{corollary}
\begin{proof} By Theorem~\ref{basset}, 
\[ K_0(x) = \int_0^{\infty} e^{-x \cosh t} \, dt  \leq \int_0^{\infty} e^{-(x/2) e^t} \, dt. \]
The change of variables $s = (x/2)e^t$ gives
\[ K_0(x) \leq \int_{x/2}^{\infty} \frac{e^{-s}}{s} \, ds \leq 1+ \log 2 +  \left\lvert \log x\right\rvert. \qedhere \]   \end{proof}

\section{Proof of lemmas and the main theorem}
\label{mainsec}

\begin{sloppypar} The following result of Tuck from \cite{tuck} gives a sufficient condition for the positivity of Fourier transforms. The proof given here is essentially the same as the one from \cite{tuck}, but is included for completeness. 
\begin{lemma}[{\cite{tuck}}] \label{tucklemma} Let $u: (0, \infty) \to \mathbb{R}$ be a differentiable convex function which is not identically zero. Assume that 
\[  \lim_{x \to +\infty} u(x) = \lim_{x \to +\infty} u'(x)  = \lim_{x \to 0^+} x u(x) =0. \]
Then 
\begin{equation} \label{improperriemann} \int_0^{\infty} \cos (x \xi) u(x) \, dx >0, \end{equation}
for any $\xi \in \mathbb{R}$ such that the two-sided improper Riemann integral in \eqref{improperriemann} converges.  \end{lemma} 
\begin{proof} That $u$ is convex and differentiable implies that $u'$ is weakly increasing on $(0, \infty)$, and therefore continuous by Darboux's theorem. The condition $\lim_{x \to +\infty} u'(x)=0$ implies that $u'(x) \leq 0$ for all $x>0$, and the condition $\lim_{x \to +\infty} u(x) =0$ then implies that $u(x) \geq 0$ for all $x>0$. Let $\xi \in \mathbb{R}$ be such that the integral in \eqref{improperriemann} converges. By scaling and by symmetry it may be assumed that $\xi = 1$ (the case $\xi = 0$ is trivial). Then
\[ \int_0^{\infty} \cos(x) u(x) \, dx = - \int_0^{\infty} \sin(x) u'(x) \, dx, \]
where the right-hand side is again a convergent two-sided improper Riemann integral. This follows by integrating by parts on $[a,b]$ with $0 < a < b < \infty$, and letting $a \to 0^+$ and $b \to \infty$, using the conditions $\lim_{x \to 0^+} x u(x) =0$ and $\lim_{x \to +\infty} u(x)=0$ to eliminate the boundary terms. It suffices to show that
\[ \int_0^{\infty} \sin(x) v(x) \, dx >0, \]
whenever $v$ is a weakly decreasing continuous function on $(0, \infty)$ with $\lim_{x \to +\infty} v(x) = 0$, such that $v$ is not identically zero and such that the above two-sided improper Riemann integral converges. By the assumption of convergence, the integral above can be written as 
\[ \sum_{k=0}^{\infty} \int_{2k\pi }^{(2k+1)\pi } \sin(x) [v(x) - v(x + \pi )] \, dx. \]
It suffices to find a single summand which is strictly positive. The function $v$ is non-constant since $v$ is not identically zero and $\lim_{x \to \infty} v(x) = 0$. Since $v$ is weakly decreasing and non-constant, there must exist a non-negative integer $k_0$ such that 
\[ v(2k_0 \pi) - v((2k_0+2) \pi ) >0; \]
where the right limit is taken for $v(0)$ and allowed to be $+\infty$ if $k_0=0$. By continuity of $v$, it follows that there exists $x_0 \in (2k_0\pi, (2k_0+1)\pi)$ such that 
\[ v(x_0) - v(x_0 + \pi )>0. \]
Again by continuity of $v$, this yields
\begin{multline*} \sum_{k=0}^{\infty} \int_{2k\pi }^{(2k+1)\pi } \sin(x) [v(x) - v(x + \pi )] \, dx \\
\geq \int_{2k_0\pi}^{(2k_0+1)\pi} \sin(x) [ v(x) - v(x + \pi) ] \, dx >0. \end{multline*}
 This finishes the proof. \end{proof} \end{sloppypar}

The following lemma is an inequality for the (2-dimensional) Euclidean Fourier transforms of (2-dimensional) Korányi kernels.

\begin{lemma} \label{koranyipotential} For $s \in (0,3)$, let
\[ f_s(x,t) = \frac{1}{(x^4 + t^2)^{s/4} }, \qquad (x,t) \in \mathbb{R}^2. \]
Then $\phi \mapsto \int \phi f_s$ defines a tempered distribution $f_s \in \mathcal{S}'(\mathbb{R}^2)$, and if $s \in (1,3)$ then the Euclidean Fourier transform of $f_s$ is a locally integrable function which satisfies
\begin{equation} \label{sthreeminuss} 0 < \widehat{f_s}(\xi_1, \xi_2) \leq C_s f_{3-s}(\xi_1, \xi_2),  \end{equation}
for some positive constant $C_s$ depending only on $s$. Moreover, if $s \in (1,3)$ then
\begin{equation} \label{fourierfunction} \widehat{f_s}(\xi_1, \xi_2) = \frac{ 2\pi^{s/4} \left\lvert \xi_2\right\rvert^{\frac{s-2}{4}}}{\Gamma(s/4)} \int_{\mathbb{R}} e^{-2\pi i x \xi_1} \left\lvert x\right\rvert^{1-\frac{s}{2} } K_{\frac{s-2}{4}}\left( 2\pi \lvert \xi_2 \rvert x^2 \right) \, dx. \end{equation}
\end{lemma}

Before the proof, a brief heuristic explanation will be given for the appearance of $f_{3-s}$ in \eqref{sthreeminuss}. If $f_s$ were a function in $L^1(\mathbb{R}^2)$, the lemma would follow by changing variables $(x,t) \mapsto (\lambda x, \lambda^2 t)$ in the integral 
\[ \widehat{f_s}(\xi_1, \xi_2) = \int_{\mathbb{R}^2} e^{-2\pi i \langle (x,t), (\xi_1, \xi_2) \rangle} f_s(x,t) \, dx \, dt, \]
using the scaling $f_s(\lambda x, \lambda^2 t) = \lambda^{-s} f_s(x,t)$ for the particular choice $\lambda = (\xi_1^4 + \xi_2^2 )^{-1/4}$, to get
\[ \widehat{f_s}(\xi_1, \xi_2)  
= f_{3-s}(\xi_1, \xi_2) \int_{\mathbb{R}^2} e^{-2\pi i \left\langle (x,t), \left(\frac{\xi_1}{\left(\xi_1^4 + \xi_2^2\right)^{1/4}}, \frac{\xi_2}{\left(\xi_1^4 + \xi_2^2\right)^{1/2}}\right)\right\rangle} f_s(x,t) \, dx \, dt. \]
If $f_s$ were in $L^1(\mathbb{R}^2)$, the integral on the right-hand side above would be a continuous function of $(\eta_1, \eta_2) = \left(\frac{\xi_1}{\left(\xi_1^4 + \xi_2^2\right)^{1/4}}, \frac{\xi_2}{\left(\xi_1^4 + \xi_2^2\right)^{1/2}}\right)$ on the compact set $\eta_1^4 + \eta_2^2 =1$, and the inequality $\left\lvert \widehat{f_s} \right\rvert \lesssim f_{3-s}$ would follow. Unfortunately, there does not seem to be an obvious way to interpret the integral above as a bounded function of $\eta$, since $f_s$ is not integrable over $\mathbb{R}^2$. A scaling property also explains the analogous formula $\widehat{k_s} = c_{n,s} k_{n-s}$ for the Riesz kernel on $\mathbb{R}^n$, but in that case the rotation invariance of the Riesz kernel makes the formula an exact equality. The lack of symmetry between the variables in $f_s$ is the reason for integrating in the $t$-variable first in the proof below.

\begin{proof}[Proof of Lemma~\ref{koranyipotential}] The first part of the proof will establish \eqref{fourierfunction}. The idea is to apply Basset's integral formula, but some analysis is necessary to justify interchanging the order of integration.  The second part of the proof will use \eqref{fourierfunction} to prove \eqref{sthreeminuss}.

The assumption that $0 < s < 3$ implies that $f_s$ is bounded outside $B(0,1)$ and locally integrable, and $f_s$ is therefore a tempered distribution. Assume now that $1 < s < 3$. Let $\psi$ be a smooth bump function on $\mathbb{R}$ such that $\psi = 1$ on $[-1,1]$ and $\psi=0$ outside $[-2,2]$. Let $\phi \in \mathcal{S}(\mathbb{R}^2)$ and for each $\epsilon>0$ let $\phi_{\epsilon}(\xi_1, \xi_2) = (1- \psi(\xi_2/ \epsilon)) \phi(\xi_1, \xi_2)$. Then  $\left\langle f_s,  \widehat{\phi} -  \widehat{\phi_{\epsilon}} \right\rangle \to 0$ as $\epsilon \to 0$. To see this, write 
\[ f_s = f_{s,1} + f_{s,2} + f_{s,3}, \]
where 
\[ f_{s,1} = f_s \chi_{ \left\{ (x,t) \in \mathbb{R}^2: x^4 + t^2 \leq 1 \right\}}, \]
\[ f_{s,2} = f_s \chi_{\left\{ (x,t) \in \mathbb{R}^2: x^4 + t^2 > 1 \text{ and } x^2 \leq \left\lvert t\right\rvert \right\}},\]
and 
\[ f_{s,3} = f_s \chi_{\left\{ (x,t) \in \mathbb{R}^2: x^4 + t^2 > 1 \text{ and } x^2 > \left\lvert t\right\rvert \right\}}. \]
Then $\left\langle f_{s,1},  \widehat{\phi} -  \widehat{\phi_{\epsilon}}  \right\rangle \to 0$ as $\epsilon \to 0$, since $f_{s,1} \in L^1(\mathbb{R}^2)$ and $ \widehat{\phi} -  \widehat{\phi_{\epsilon}}  \to 0$ in $L^{\infty}(\mathbb{R}^2)$. 

By an integration by parts in the $\xi_2$ variable, 
\begin{multline*} \left\langle f_{s,2},  \widehat{\phi} -  \widehat{\phi_{\epsilon}}  \right\rangle =\\ \int_{\mathbb{R}^2} \frac{ f_{s,2}(x,t)}{2\pi i t} \int_{\mathbb{R}^2} e^{-2\pi i \langle (x,t), \xi \rangle } \left[ \epsilon^{-1}  \psi'(\xi_2/\epsilon) \phi( \xi) + \psi(\xi_2/\epsilon) \partial_2 \phi(\xi ) \right] \, d\xi \, dx \, dt. \end{multline*} 
The functions
\begin{equation} \label{tricky} \int_{\mathbb{R}^2} e^{-2\pi i \langle (x,t), \xi \rangle } \left[ \epsilon^{-1}  \psi'(\xi_2/\epsilon) \phi( \xi) + \psi(\xi_2/\epsilon) \partial_2 \phi(\xi ) \right] \, d\xi \end{equation}
are uniformly bounded in $L^{\infty}(\mathbb{R}^2)$, and converge to zero uniformly on compact subsets of $\mathbb{R}^2$, as $\epsilon \to 0$. The convergence to zero of the second term in \eqref{tricky} follows from the dominated convergence theorem. Convergence to zero of the first term in \eqref{tricky} follows by writing 
\begin{multline*}  \int_{\mathbb{R}^2} e^{-2\pi i \langle (x,t), \xi \rangle }  \epsilon^{-1}  \psi'(\xi_2/\epsilon) \phi( \xi ) \, d\xi \\
=  \int_{\mathbb{R}^2} e^{-2\pi i \langle (x,t), (\xi_1, \epsilon \eta ) \rangle }    \psi'(\eta) \left[\phi( \xi_1, \epsilon \eta )  - \phi(\xi_1, 0) \right] \, d\xi_1 \, d\eta  \\
+ \int_{\mathbb{R}^2} \left[ e^{-2\pi i \langle (x,t), (\xi_1, \epsilon \eta ) \rangle } - e^{-2\pi i \langle (x,t), (\xi_1, 0 ) \rangle } \right]   \psi'(\eta) \phi(\xi_1, 0) \, d\xi_1 \, d\eta, \end{multline*}
and applying the dominated convergence theorem. The function $\frac{ f_{s,2}(x,t)}{2\pi i t}$ is in $L^1(\mathbb{R}^2)$ since $s>1$. Hence $\left\langle f_{s,2},  \widehat{\phi} -  \widehat{\phi_{\epsilon}}  \right\rangle \to 0$ as $\epsilon \to 0$. 

For the third function, integrating by parts twice in the $\xi_1$ variable gives that $\left\langle f_{s,3},  \widehat{\phi} -  \widehat{\phi_{\epsilon}}  \right\rangle \to 0$ as $\epsilon \to 0$. This shows that $\left\langle f_s,  \widehat{\phi} -  \widehat{\phi_{\epsilon}} \right\rangle \to 0$ as $\epsilon \to 0$. Hence
\begin{align} \notag &\int_{\mathbb{R}^2}  \frac{\widehat{\phi}(x,t)}{(x^4 + t^2)^{s/4} }  \, dx \, dt\\
\notag &\quad = \lim_{\epsilon \to 0 }\int_{\mathbb{R}^2}  \frac{\widehat{\phi_{\epsilon}}(x,t)}{(x^4 + t^2)^{s/4} }  \, dx \, dt\\
\label{stop3} &\quad = \lim_{\epsilon \to 0} \lim_{M \to \infty} \lim_{N \to \infty} \int_{-M}^M \int_{\mathbb{R}^2} e^{-2\pi i x \xi_1} \phi_{\epsilon}(\xi) \int_{-N}^N  \frac{e^{-2\pi i t \xi_2 }}{(x^4 + t^2)^{s/4} }  \, dt \, d\xi \, dx. \end{align}
 For any $x, \xi_2 \in \mathbb{R}$ both nonzero, and any $s>0$, Theorem~\ref{basset} gives
\begin{equation} \label{bassets} \int_{-\infty}^{\infty} \frac{e^{-2\pi i t \xi_2 }}{(x^4 + t^2)^{s/4} }  \, dt =  \frac{ 2\pi^{s/4}}{\Gamma(s/4) } \left\lvert x\right\rvert^{1-\frac{s}{2} } \left\lvert \xi_2\right\rvert^{\frac{s-2}{4} }K_{\frac{s-2}{4}}\left( 2\pi x^2 \left\lvert \xi_2\right\rvert \right), \end{equation}
where the integral is an improper Riemann integral.  By the second mean value theorem for integrals (or an integration by parts),
\begin{equation} \label{2ndmvt} \left\lvert \int_{-\infty}^{\infty} \frac{e^{-2\pi i t \xi_2 } }{(x^4 + t^2)^{s/4} } \, dt - \int_{-N}^{N} \frac{e^{-2\pi i t \xi_2 }}{(x^4 + t^2)^{s/4} }  \, dt  \right\rvert   \lesssim \frac{1}{\left\lvert \xi_2\right\rvert\left(  \left\lvert x\right\rvert^s + N^{s/2} \right)}, \end{equation}
for any $N >0$. Hence, by three applications of the dominated convergence theorem,
\begin{align*} \eqref{stop3} &= \lim_{\epsilon \to 0} \lim_{M \to \infty}  \int_{-M}^M \int_{\mathbb{R}^2} e^{-2\pi i x \xi_1} \phi_{\epsilon}(\xi) \int_{-\infty}^{\infty}  \frac{e^{-2\pi i t \xi_2 } }{(x^4 + t^2)^{s/4} } \, dt \, d\xi \, dx \\
&\quad =   \int_{\mathbb{R}} \int_{\mathbb{R}^2} e^{-2\pi i x \xi_1} \phi(\xi) \int_{-\infty}^{\infty}  \frac{e^{-2\pi i t \xi_2 }}{(x^4 + t^2)^{s/4} }  \, dt \, d\xi \, dx. \end{align*} 
The first application used \eqref{2ndmvt} to get the dominating function 
\[ \left\lvert \phi_{\epsilon}(\xi)\right\rvert \left[ \frac{1}{\left\lvert \xi_2\right\rvert\left(  \left\lvert x\right\rvert^s + 1 \right)} + \left\lvert x\right\rvert^{1-\frac{s}{2} } \left\lvert \xi_2\right\rvert^{\frac{s-2}{4} }K_{\frac{s-2}{4}}\left( 2\pi x^2 \left\lvert \xi_2\right\rvert \right) \right], \]
where $(x, \xi) \in [-M,M] \times \mathbb{R}^2$, whilst the second and third applications used the dominating function
\[ \left\lvert \phi(\xi) \right\rvert\left\lvert x\right\rvert^{1-\frac{s}{2} } \left\lvert \xi_2\right\rvert^{\frac{s-2}{4} }K_{\frac{s-2}{4}}\left( 2\pi x^2 \left\lvert \xi_2\right\rvert \right), \quad (x, \xi) \in \mathbb{R}^3, \]
which is integrable on $\mathbb{R}^3$ since $1 < s < 3$; by changing variables and considering the behaviour of $K_{\frac{s-2}{4}}$ for small and large arguments. More precisely, for $s \neq 2$ the series definition of $K_{\frac{s-2}{4}}$ gives the behaviour of $K_{\frac{s-2}{4}}$ for small arguments, whilst Corollary~\ref{besselzero} controls the behaviour of $K_0$ for small arguments. Corollary~\ref{asymptotics} controls the asymptotic behaviour of $K_{\frac{s-2}{4}}$ for large arguments. By Fubini, \eqref{bassets}, and a change of variables,
\begin{multline*} \int_{\mathbb{R}^2}  \frac{\widehat{\phi}(x,t)}{(x^4 + t^2)^{s/4} }  \, dx \, dt \\
= \int_{\mathbb{R}^2} \phi(\xi) \left[\frac{ 2\pi^{s/4} \left\lvert \xi_2\right\rvert^{\frac{s-2}{4 }}}{\Gamma(s/4)} \int_{\mathbb{R}} e^{-2\pi i x \xi_1} \left\lvert x\right\rvert^{1-\frac{s}{2} } K_{\frac{s-2}{4}}\left( 2\pi \lvert\xi_2\rvert x^2 \right) \, dx \right]\, d\xi \\
= \frac{ 2\pi^{s/4}}{\Gamma(s/4) }  \int_{\mathbb{R}^2} \phi(\xi) \left[\left\lvert \xi_2\right\rvert^{\frac{s-3}{2} } \int_{\mathbb{R}} e^{-2\pi i x \frac{\xi_1}{\left\lvert \xi_2\right\rvert^{1/2}}} \left\lvert x\right\rvert^{1-\frac{s}{2} } K_{\frac{s-2}{4}}\left( 2\pi x^2 \right) \, dx \right] \, d\xi. \end{multline*} 
This proves \eqref{fourierfunction}. It remains to show that for $\xi_2 \neq 0$,
\begin{equation} \label{basepoint}  \left\lvert \xi_2\right\rvert^{\frac{s-3}{2} } \int_{-\infty}^{\infty} e^{-2\pi i x \frac{\xi_1}{\left\lvert \xi_2\right\rvert^{1/2}}} \left\lvert x\right\rvert^{1-\frac{s}{2} } K_{\frac{s-2}{4}}\left( 2\pi x^2 \right) \, dx \lesssim f_{3-s}(\xi_1,\xi_2), \end{equation}
and that the left-hand side of \eqref{basepoint} is strictly positive. The function 
\begin{equation} \label{decfunction}  - \frac{d}{dx} \left(\left\lvert x\right\rvert^{1-\frac{s}{2} } K_{\frac{s-2}{4}}\left( 2\pi x^2 \right) \right)  = \left[ 4\pi \left\lvert x \right\rvert^{1-s} \right] \cdot  \left[ \left\lvert x\right\rvert^{2\left(\frac{s+2}{4}\right) } K_{\frac{s+2}{4} }(2\pi x^2) \right], \end{equation}
is decreasing on $(0, \infty)$ since it is a product of positive, decreasing functions; by Proposition~\ref{derivativeidentity}, Theorem~\ref{basset}, and the assumption that $s>1$. More precisely, the derivative was calculated by substituting $u=2\pi x^2$, calculating the derivative with respect to $u$ using the second identity in Proposition~\ref{derivativeidentity}, and then applying the chain rule. This verifies the convexity condition in Lemma~\ref{tucklemma}. The other conditions of Lemma~\ref{tucklemma} follow by considering the behaviour of $K_{\frac{s-2}{4}}$ for small and large arguments; using the series definition, Proposition~\ref{derivativeidentity}, Corollary~\ref{asymptotics}, Corollary~\ref{besselzero}, and the assumption that $s< 3$. By Lemma~\ref{tucklemma}, the left-hand side of \eqref{basepoint} is strictly positive for $s \in (1,3)$, and therefore $\widehat{f_s} >0$. 

If $\left\lvert \xi_2\right\rvert^{1/2} \geq \left\lvert \xi_1\right\rvert$, the inequality in \eqref{basepoint} is immediate since the integrand has $L^1$ norm $\lesssim 1$. This covers the case $\left\lvert \xi_2\right\rvert^{1/2} \geq \left\lvert \xi_1\right\rvert$.

Henceforth suppose that $\left\lvert \xi_2\right\rvert^{1/2} < \left\lvert \xi_1\right\rvert$. By symmetry and an integration by parts,
\begin{multline} \label{STAR}  \int_{-\infty}^{\infty} e^{-2\pi i x \frac{\xi_1}{\left\lvert \xi_2\right\rvert^{1/2}}} \left\lvert x\right\rvert^{1-\frac{s}{2} } K_{\frac{s-2}{4}}\left( 2\pi x^2 \right) \, dx \\
= \frac{-\left\lvert \xi_2\right\rvert^{1/2}}{\pi \xi_1} \int_{0}^{\infty} \sin\left(2\pi  x \frac{\xi_1}{\left\lvert \xi_2\right\rvert^{1/2}}\right)  \frac{d}{dx} \left( \left\lvert x\right\rvert^{1-\frac{s}{2} } K_{\frac{s-2}{4}}\left( 2\pi x^2 \right) \right) \, dx. \end{multline}
More precisely, the fact that $\left\lvert x\right\rvert^{1-\frac{s}{2} } K_{\frac{s-2}{4}}\left( 2\pi x^2 \right)$ is even means that only the cosine of the exponential in \eqref{STAR} contributes, and the interval of integration can be changed to $(0, \infty)$ resulting in a factor of 2. The integration by parts is first done on $(\delta, 1/\delta)$ for some $\delta>0$, and then $\delta$ is taken to zero, with the boundary terms vanishing in the limit since $s<3$ and due to the behaviour of $K_{\frac{s-2}{4}}$ for small and large arguments (again using the series definition, Corollary~\ref{asymptotics} and Corollary~\ref{besselzero}).  
By \eqref{decfunction}, 
\begin{align}\notag &\left\lvert \int_0^{\frac{\left\lvert \xi_2\right\rvert^{1/2}}{\left\lvert \xi_1\right\rvert}} \sin\left(2\pi  x \frac{\xi_1}{\left\lvert \xi_2\right\rvert^{1/2}}\right) \frac{d}{dx}\left(  \left\lvert x\right\rvert^{1-\frac{s}{2} } K_{\frac{s-2}{4}}\left( 2\pi x^2 \right) \right) \, dx \right\rvert \\
\notag&\quad \lesssim  \frac{\left\lvert \xi_1\right\rvert }{ \left\lvert \xi_2\right\rvert^{1/2} }   \int_0^{\frac{\left\lvert \xi_2\right\rvert^{1/2}}{\left\lvert \xi_1\right\rvert}} \left\lvert x\right\rvert^{3-\frac{s}{2} } K_{\frac{s+2}{4}}\left( 2\pi x^2 \right)   \, dx \\
\notag &\quad \lesssim   \frac{\left\lvert \xi_1\right\rvert }{ \left\lvert \xi_2\right\rvert^{1/2} } \int_0^{\frac{\left\lvert \xi_2\right\rvert^{1/2}}{\left\lvert \xi_1\right\rvert}} \left\lvert x\right\rvert^{2-s }   \, dx \\
\label{smallxsecond} &\quad \lesssim \left( \frac{ \left\lvert \xi_2\right\rvert^{1/2} }{\left\lvert \xi_1\right\rvert } \right)^{2-s}. \end{align}
It remains to bound the part of the integral over $x>  \frac{\left\lvert \xi_2\right\rvert^{1/2}}{\left\lvert \xi_1\right\rvert}$. Since the right-hand side of \eqref{decfunction} is decreasing, the second mean value theorem for integrals (or an integration by parts) can be applied to get
\begin{align}\notag &\left\lvert \int_{\frac{\left\lvert \xi_2\right\rvert^{1/2}}{\left\lvert \xi_1 \right\rvert}}^{\infty} \sin\left(2\pi  x \frac{\xi_1}{\left\lvert \xi_2\right\rvert^{1/2}}\right) \left( \frac{d}{dx} \left\lvert x\right\rvert^{1-\frac{s}{2} } K_{\frac{s-2}{4}}\left( 2\pi x^2 \right) \right) \, dx \right\rvert \\
\notag &\quad \lesssim \left( \frac{\left\lvert \xi_2\right\rvert^{1/2}}{\left\lvert \xi_1 \right\rvert} \right)^{3-\frac{s}{2} } K_{\frac{s+2}{4}}\left( \frac{ 2\pi \left\lvert \xi_2\right\rvert}{\left\lvert \xi_1 \right\rvert^2}  \right) \\
\label{mvtsecond} &\quad\lesssim \left( \frac{ \left\lvert \xi_2\right\rvert^{1/2} }{\left\lvert \xi_1\right\rvert } \right)^{2-s}. \end{align} 
Substituting \eqref{smallxsecond} and \eqref{mvtsecond} back into \eqref{STAR} shows that \eqref{basepoint} holds for all $s \in (1,3)$. This finishes the proof.  \end{proof}
Let $\iota: \mathbb{R}^2 \to \mathbb{R}^2$ be the inverse map $(x,t) \mapsto (-x,-t)$, and let $\mathcal{F}^{-1}$ be the Euclidean inverse Fourier transform. The proof of the following lemma follows \cite[p.~39]{mattila}. 
\begin{lemma} \label{lemmaplanch} If $s \in (1,3)$ and $\mu$ is a finite compactly supported Borel measure on $\mathbb{R}^2$, then 
\[ \int f_s \, d\left( \iota_{\#} \mu \ast \mu \right) \leq \int_{\mathbb{R}^2} \widehat{f_s} \mathcal{F}^{-1} \left( \iota_{\#} \mu \ast \mu \right). \] \end{lemma} 
\begin{proof} Let $\phi$ be a smooth, even, non-negative bump function on $\mathbb{R}^2$ with $\int \phi = 1$. For each $\epsilon >0$, let $\phi_{\epsilon}(x) = \epsilon^{-2} \phi(x/\epsilon)$, and let $\mu_{\epsilon} = \mu \ast \phi_{\epsilon}$. Since $\widehat{f_s} \geq 0$ (by Lemma~\ref{koranyipotential}), it may be assumed that $\widehat{f_s} \mathcal{F}^{-1} \left( \iota_{\#} \mu \ast \mu \right)$ is absolutely integrable. By the dominated convergence theorem, a change of variables, and Fatou's lemma,
\begin{align*} &\int_{\mathbb{R}^2} \widehat{f_s} \mathcal{F}^{-1} \left( \iota_{\#} \mu \ast \mu \right) \\
&\quad=  \lim_{\epsilon \to 0}  \int_{\mathbb{R}^2} \widehat{f_s} \mathcal{F}^{-1} \left( \iota_{\#} \mu_{\epsilon} \ast \mu_{\epsilon} \right) \\
&\quad=  \lim_{\epsilon \to 0} \int f_s \, d\left( \iota_{\#} \mu_{\epsilon} \ast \mu_{\epsilon} \right) \\
&\quad =  \lim_{\epsilon \to 0} \int \int \int_{\mathbb{R}^2} \int_{\mathbb{R}^2} f_s(x'-y' ) \phi_{\epsilon}(x-x') \phi_{\epsilon}(y-y') \, dx' \, dy' \, d\mu(x) \, d\mu(y) \\
&\quad = \lim_{\epsilon \to 0} \int \int \int_{\mathbb{R}^2} \int_{\mathbb{R}^2} f_s(x-y - \epsilon w + \epsilon z ) \phi(w) \phi(z) \, dw \, dz \, d\mu(x) \, d\mu(y) \\
&\quad\geq \int f_s \, d\left( \iota_{\#} \mu \ast \mu \right). \qedhere \end{align*}  \end{proof}

\begin{proof}[Proof of Theorem~\ref{mainthm}] By the scaling $(z,t) \mapsto (\lambda z, \lambda^2 t)$, it may be assumed that $A$ is contained in the Korányi unit ball around the origin. It may also be assumed that 
\begin{equation} \label{taxis} \dim\left( A \setminus \left\{ (0,t) \in \mathbb{C} \times \mathbb{R} \right\} \right)  = \dim A, \end{equation}
since otherwise the theorem is immediate. 

Let $\alpha$ be such that $1 < \alpha <  \min\{3,\dim A\}$, and suppose that $1 < s < (1+\alpha)/2$. By Frostman's lemma, it suffices to prove that for any $\epsilon >0$, there exists a Borel probability measure $\mu$ supported on $A$, and a Borel set $E \subseteq [0,\pi)$ with $m([0,\pi) \setminus E) \leq \epsilon$, such that
\[ \int_E \int_{\mathbb{V}_{\theta}^{\perp}} \int_{\mathbb{V}_{\theta}^{\perp}} d_{\mathbb{H}}((z,t), (\zeta, \tau) )^{-s} \, d\left(P_{\mathbb{V}_{\theta}^{\perp}\#}\mu\right)(z,t) \, d\left(P_{\mathbb{V}_{\theta}^{\perp}\#}\mu\right)(\zeta, \tau) \, d\theta < \infty. \]

By \eqref{taxis}, there is a number $c=c(A, \alpha)>0$ such that for any $\epsilon>0$, there exists $\theta_0 \in [0, \pi)$, and a Borel probability measure $\mu$ on $A$, supported in a Korányi ball of radius $1/2$, with 
\begin{equation} \label{frostman} c_{\alpha}(\mu) = \sup_{ \substack{ (z,t) \in \mathbb{H} \\ r >0 } } \frac{ \mu\left( B_{\mathbb{H}}\left((z,t), r\right) \right) }{r^{\alpha}} < \infty, \end{equation} 
 such that
\begin{equation} \label{radiusbound} \left\lvert z\right\rvert > c, \end{equation}
for all $(z,t) \in \supp \mu$, and such that  either
\begin{equation} \label{alt1} \lvert \arg z - \theta_0\rvert_{\bmod 2\pi} < \epsilon^3 \quad \text{ for all } (z,t) \in \supp \mu, \end{equation}
or
\begin{equation} \label{alt2} \lvert \arg z + \pi - \theta_0\rvert_{\bmod 2\pi} < \epsilon^3 \quad \text{ for all } (z,t) \in \supp \mu. \end{equation}

Let $\epsilon >0$ be given, assuming $\epsilon < 1/100$ without loss of generality, let $c$, $\theta_0$ and $\mu$ be as described above, and let 
\begin{equation} \label{Edefn} E = \left\{ \theta \in [0, \pi) : \left\lvert \theta - \theta_0 - \pi/2  \right\rvert_{\bmod \pi} > \epsilon/2 \right\}, \end{equation}
which satisfies $m([0,\pi) \setminus E) \leq \epsilon$. By rotating $\mathbb{V}_{\theta}^{\perp}$ to $\mathbb{R}^2$ and applying Lemma~\ref{lemmaplanch},
\begin{align} \notag &\int_E \int_{\mathbb{V}_{\theta}^{\perp}} \int_{\mathbb{V}_{\theta}^{\perp}}  d_{\mathbb{H}}((z,t), (\zeta, \tau) )^{-s} \, d\left(P_{\mathbb{V}_{\theta}^{\perp}\#}\mu\right)(z,t) \, d\left(P_{\mathbb{V}_{\theta}^{\perp}\#}\mu\right)(\zeta, \tau) \, d\theta \\
\notag &\quad = \int_E \int_{\mathbb{V}_{\theta}^{\perp}}  f_s(\left\lvert z\right\rvert,t) \, d\left(\iota_{\#} P_{\mathbb{V}_{\theta}^{\perp}\#} \mu \ast  P_{\mathbb{V}_{\theta}^{\perp}\#}\mu \right)(z,t) \, d\theta \\
\label{plancherel} &\quad \leq \int_E \int_{\mathbb{R}^2}  \widehat{f_s}(r, \rho)\mathcal{F}^{-1}\left( \iota_{\#}P_{\mathbb{V}_{\theta}^{\perp}\#} \mu \ast P_{\mathbb{V}_{\theta}^{\perp}\#}\mu \right)\left(rie^{i\theta}, \rho\right)  \, dr \, d\rho \, d\theta, \end{align}
where $\iota$ is the inverse map $(z,t) \mapsto (-z,-t)$, and the convolution above is Euclidean convolution (which equals Heisenberg convolution on vertical subgroups of $\mathbb{H}^1$). By Lemma~\ref{koranyipotential}, 
\[ \eqref{plancherel} \lesssim \int_{\mathbb{R}^2}  \left( r^4 + \rho^2 \right)^{(s-3)/4}  \int_E \mathcal{F}^{-1}\left( \iota_{\#} P_{\mathbb{V}_{\theta}^{\perp}\#} \mu \ast P_{\mathbb{V}_{\theta}^{\perp}\#}\mu \right)\left(rie^{i \theta}, \rho\right)   \, d\theta \, dr \, d\rho, \]
where $\mathcal{F}^{-1}\left( \iota_{\#} P_{\mathbb{V}_{\theta}^{\perp}\#} \mu \ast P_{\mathbb{V}_{\theta}^{\perp}\#}\mu \right)$ is non-negative by the convolution theorem. 

Choose $\delta>0$ such that $\delta < ((1+\alpha)/2 -s)/100$. The part of the above integral over $(r^4 + \rho^2)^{1/4} \leq 100$ is finite since $s >0$, and since the inverse Fourier transform of a probability measure is bounded by 1 in sup norm. The remaining integral over $(r^4 + \rho^2)^{1/4} > 100$ can be dyadically partitioned into a sum over sets $B_j$ where $(r^4 + \rho^2)^{1/4} \sim 2^j$ and $j \geq 0$. For each $j$, since $\mathcal{F}^{-1}\left( \iota_{\#} P_{\mathbb{V}_{\theta}^{\perp}\#} \mu \ast P_{\mathbb{V}_{\theta}^{\perp}\#}\mu \right)$ is non-negative, the factor $(r^4 + \rho^2)^{(s-3)/4}$ can be changed to $2^{-j(3-s)}$, and then the sets $B_j$ can be enlarged to
\[ A_j = \left\{ (\rho, \theta, r) :  \theta \in E, \quad \lvert\rho\rvert \leq 2^{2j}, \quad \lvert r\rvert \leq 2^j \right\}, \]
without any need for modulus signs. By then using Fubini to swap the integral defining $\mathcal{F}^{-1}\left( \iota_{\#} P_{\mathbb{V}_{\theta}^{\perp}\#} \mu \ast P_{\mathbb{V}_{\theta}^{\perp}\#}\mu \right)$ with all of the others, it suffices to show that
\begin{multline*} \int \int \left\lvert \int_{A_j}  e^{2\pi i \left\langle \left(rie^{i \theta}, \rho\right), \left(z-\zeta, t-\tau +2 \omega\left( \pi_{V_{\theta}}(z), z \right) - 2 \omega\left( \pi_{V_{\theta}}(\zeta), \zeta \right)\right) \right\rangle}  \, d\rho \, d\theta \, dr \right\rvert \\
d\mu(\zeta,\tau) \, d\mu(z,t) \lesssim 2^{j \left( 3-s-\delta\right)}, \end{multline*}
for any $j \geq 0$, where 
\[ A_j = \left\{ (\rho, \theta, r) :  \theta \in E, \quad \lvert\rho\rvert \leq 2^{2j}, \quad \lvert r\rvert \leq 2^j \right\}. \]
 Let $j \geq 0$ be given. Since $\mu$ is a probability measure, it is enough to show that for any $(z,t) \in \supp \mu$, 
\begin{multline*} \int \left\lvert \int_{A_j}  e^{2\pi i \left\langle \left(rie^{i \theta}, \rho\right), \left(z-\zeta, t-\tau +2 \omega\left( \pi_{V_{\theta}}(z), z \right) - 2 \omega\left( \pi_{V_{\theta}}(\zeta), \zeta \right)\right) \right\rangle}  \, d\rho \, d\theta \, dr \right\rvert d\mu(\zeta,\tau) \\
\lesssim 2^{j \left( 3-s-\delta\right)}. \end{multline*} 
Let $(z,t) \in \supp \mu$ be given. A trivial upper bound for the inner integral is $2^{3j}$, so using $\delta < (\alpha-s)/100$ and the Frostman condition \eqref{frostman} on $\mu$ gives
\begin{multline*} \int_{B_{\mathbb{H}}((z,t), 2^{-j} ) } \left\lvert \int_{A_j}  e^{2\pi i \left\langle \left(rie^{i \theta}, \rho\right), \left(z-\zeta, t-\tau +2 \omega\left( \pi_{V_{\theta}}(z), z \right) - 2 \omega\left( \pi_{V_{\theta}}(\zeta), \zeta \right)\right) \right\rangle}  \, d\rho \, d\theta \, dr \right\rvert \\
 d\mu(\zeta,\tau) \lesssim 2^{j \left( 3-s-\delta\right)}, \end{multline*}
Therefore, it suffices to show that
\begin{multline} \label{ksum} \sum_{k=0}^{j} \int_{B_{\mathbb{H}}\left((z,t), 2^{-k}\right) \setminus B_{\mathbb{H}}\left((z,t), 2^{-(k+1)}\right)} \\
\left\lvert \int_{A_j}  e^{2\pi i \left\langle \left(rie^{i \theta}, \rho\right), \left(z-\zeta, t-\tau +2 \omega\left( \pi_{V_{\theta}}(z), z \right) - 2 \omega\left( \pi_{V_{\theta}}(\zeta), \zeta \right)\right) \right\rangle}  \, d\rho \, d\theta \, dr \right\rvert \\
d\mu(\zeta,\tau) \lesssim 2^{j \left( 3-s-\delta\right)}, \end{multline}
Fix $k \in \{0, \dotsc, j\}$ and $(\zeta, \tau)\in \supp \mu$ with $2^{-(k+1)} \leq d_{\mathbb{H}}((z,t), (\zeta, \tau)) \leq 2^{-k}$. It will be shown that 
\begin{multline} \label{claimedbound} \left\lvert \int_{A_j}  e^{2\pi i \left\langle \left(rie^{i \theta}, \rho\right), \left(z-\zeta, t-\tau +2 \omega\left( \pi_{V_{\theta}}(z), z \right) - 2 \omega\left( \pi_{V_{\theta}}(\zeta), \zeta \right)\right) \right\rangle}  \, d\rho \, d\theta \, dr \right\rvert \\
\lesssim j\min\left\{ 2^{j+3k}, 2^{2j+k}\right\}, \end{multline} 
which will be enough to prove \eqref{ksum}. To see that \eqref{claimedbound} implies \eqref{ksum}, assume that \eqref{claimedbound} holds. Then substituting \eqref{claimedbound} into \eqref{ksum} gives 
\begin{align*} &\sum_{k=0}^j \int_{B_{\mathbb{H}}\left((z,t), 2^{-k}\right) \setminus B_{\mathbb{H}}\left((z,t), 2^{-(k+1)}\right)} \\
&\qquad \left\lvert \int_{A_j}  e^{2\pi i \left\langle \left(rie^{i \theta}, \rho\right), \left(z-\zeta, t-\tau +2 \omega\left( \pi_{V_{\theta}}(z), z \right) - 2 \omega\left( \pi_{V_{\theta}}(\zeta), \zeta \right)\right) \right\rangle}  \, d\rho \, d\theta \, dr \right\rvert d\mu(\zeta,\tau) \\
&\quad \lesssim  j\sum_{k \in [0,j/2]} 2^{j+3k-k \alpha}  + j\sum_{k \in [j/2,j]} 2^{2j+k-k \alpha} \\
&\quad \lesssim  j2^{j\left( \frac{5-\alpha}{2} \right) }  \\
&\quad \lesssim 2^{ j\left(3-s  - \delta\right) }, \end{align*}
since $1< \alpha < 3$ and $0 < \delta < \left( \frac{1+\alpha}{2} -s\right)/100$. This proves that \eqref{ksum} holds conditionally on \eqref{claimedbound}, which (as explained previously) implies the theorem.

It remains to prove \eqref{claimedbound}. If $\left\lvert z-\zeta\right\rvert \leq 2^{-2k}/100$, then $\left\lvert t-\tau+2 \omega(z, \zeta)\right\rvert \geq 2^{-2k}/10$, and hence 
\[ \left\lvert t-\tau +2 \omega\left( \pi_{V_{\theta}}(z), z \right) - 2 \omega\left( \pi_{V_{\theta}}(\zeta), \zeta \right)\right\rvert \gtrsim 2^{-2k}, \]
for all $\theta \in [0, \pi)$, by the identity
\[  \omega\left( \pi_{V_{\theta}}(z), z \right) -  \omega\left( \pi_{V_{\theta}}(\zeta), \zeta \right) =  \omega(z, \zeta) +  \omega\left( \pi_{V_{\theta}}(z+\zeta), z-\zeta \right). \]
It follows that 
\begin{equation} \label{vertcase} \left\lvert \int_{A_j}  e^{2\pi i \left\langle \left(rie^{i \theta}, \rho\right), \left(z-\zeta, t-\tau +2 \omega\left( \pi_{V_{\theta}}(z), z \right) - 2 \omega\left( \pi_{V_{\theta}}(\zeta), \zeta \right)\right) \right\rangle}  \, d\rho \, d\theta \, dr \right\rvert \\ \lesssim 2^{j + 2k}. \end{equation}
This implies \eqref{claimedbound} in this case, so it will henceforth be assumed that $\left\lvert z-\zeta\right\rvert > 2^{-2k}/100$.  Let $p= (z-\zeta)/\lvert z-\zeta \rvert$, $q = (z+\zeta)/\lvert z+\zeta \rvert$ and let 
\[ E^{(1)} = \left\{ \theta \in E: \left\lvert\left\langle p, i e^{i \theta} \right\rangle \right\rvert < \epsilon^3 \right\}, \quad E^{(2)}= E \setminus E^{(1)}.  \]
For $l \in \{1,2\}$ let $A_j^{(l)}= \left\{ (\rho,\theta, r) \in A_j : \theta \in E^{(l)}\right\}$. Then
\begin{multline} \label{twoparts} \left\lvert \int_{A_j}  e^{2\pi i \left\langle \left(rie^{i \theta}, \rho\right), \left(z-\zeta, t-\tau +2 \omega\left( \pi_{V_{\theta}}(z), z \right) - 2 \omega\left( \pi_{V_{\theta}}(\zeta), \zeta \right)\right) \right\rangle}  \, d\rho \, d\theta \, dr \right\rvert \\
\leq  \left\lvert \int_{A_j^{(1)}}  e^{2\pi i \left\langle \left(rie^{i \theta}, \rho\right), \left(z-\zeta, t-\tau +2 \omega\left( \pi_{V_{\theta}}(z), z \right) - 2 \omega\left( \pi_{V_{\theta}}(\zeta), \zeta \right)\right) \right\rangle}  \, d\rho \, d\theta \, dr \right\rvert \\
+ \left\lvert \int_{A_j^{(2)}}  e^{2\pi i \left\langle \left(rie^{i \theta}, \rho\right), \left(z-\zeta, t-\tau +2 \omega\left( \pi_{V_{\theta}}(z), z \right) - 2 \omega\left( \pi_{V_{\theta}}(\zeta), \zeta \right)\right) \right\rangle}  \, d\rho \, d\theta \, dr \right\rvert. \end{multline}
Integrating the exponential above in the $r$-variable would result in a factor of $\langle i e^{i \theta}, z-\zeta \rangle^{-1}$, which is bounded on $A_j^{(2)}$ by definition of $E^{(2)}$. The integral over $A_j^{(1)}$ must be treated differently, since for $\theta \in E^{(1)}$ there is potentially no cancellation occurring in the $r$-integration, but this will be made up for by better cancellation in the $\rho$-integration (on average). 

The second integral in the right-hand side of \eqref{twoparts} satisfies
\begin{align} \notag &\left\lvert \int_{A_j^{(2)}}  e^{2\pi i \left\langle \left(rie^{i \theta}, \rho\right), \left(z-\zeta, t-\tau +2 \omega\left( \pi_{V_{\theta}}(z), z \right) - 2 \omega\left( \pi_{V_{\theta}}(\zeta), \zeta \right)\right) \right\rangle}  \, d\rho \, d\theta \, dr \right\rvert \\
\notag &\quad \leq \int_{\left\{ \theta \in E^{(2)} : \left\lvert t- \tau + 2\omega\left( \pi_{V_{\theta}}(z), z\right) - 2\omega\left(\pi_{V_{\theta}}(\zeta), \zeta \right) \right\rvert  < 2^{-{2j}} \right\}} \\
\notag &\qquad \left\lvert \int_{\pi\left(A_j\right)} e^{2\pi i \left\langle \left(rie^{i \theta}, \rho\right), \left(z-\zeta, t-\tau +2 \omega\left( \pi_{V_{\theta}}(z), z \right) - 2 \omega\left( \pi_{V_{\theta}}(\zeta), \zeta \right)\right) \right\rangle}  \, d\rho \, dr \right\rvert \, d\theta \\
\notag &\qquad + \sum_{l=-\infty}^{2j}  \int_{\left\{ \theta \in E^{(2)} : 2^{-(l+1)} \leq \left\lvert t- \tau + 2\omega\left( \pi_{V_{\theta}}(z), z\right) - 2\omega\left(\pi_{V_{\theta}}(\zeta), \zeta \right) \right\rvert  < 2^{-l} \right\} } \\
\label{pause} &\qquad \quad \left\lvert \int_{\pi\left(A_j\right)} e^{2\pi i \left\langle \left(rie^{i \theta}, \rho\right), \left(z-\zeta, t-\tau +2 \omega\left( \pi_{V_{\theta}}(z), z \right) - 2 \omega\left( \pi_{V_{\theta}}(\zeta), \zeta \right)\right) \right\rangle}  \, d\rho \, dr \right\rvert \, d\theta. \end{align}
where 
\[ \pi\left(A_j\right) = \left\{ (\rho, r) \in \mathbb{R}^2:  \quad \lvert \rho\rvert \leq 2^{2j}, \quad \lvert r\rvert \leq 2^j  \right\} \]
 is the projection of $A_j$ onto the $(\rho,r)$-plane. A similar calculation to \cite[Lemma~2.3]{me} (following \cite[Section 4]{balogh2} and \cite[Lemma~3.5]{fassler}) gives that the function 
\begin{multline} \label{Fdefn}  F(\theta) = t-\tau +2 \omega\left( \pi_{V_{\theta}}(z), z \right) - 2 \omega\left( \pi_{V_{\theta}}(\zeta), \zeta \right) \\
= t-\tau + 2 \omega\left( \pi_{V_{\theta}}(z-\zeta), z+\zeta \right) - 2 \omega(z, \zeta), \end{multline}
satisfies 
\begin{equation} \label{meanvalueprep} 2^{-4k} \lesssim \left\lvert z-\zeta\right\rvert^2\left\lvert z+\zeta\right\rvert^2 = \left\lvert \frac{F'(\theta)}{2} \right\rvert^2 + \left\lvert \frac{F''(\theta)}{4} \right\rvert^2, \end{equation}
where the lower bound in \eqref{meanvalueprep} uses that $\lvert z+\zeta\rvert \gtrsim 1$, which follows from \eqref{radiusbound} and \eqref{alt1}--\eqref{alt2}. By the mean value theorem, the equality in \eqref{meanvalueprep} implies that the set $\left\{ \theta \in [0, \pi) : \left\lvert F'(\theta) \right\rvert < (\lvert z-\zeta\rvert \lvert z+\zeta\rvert)/100 \right\}$ is a union of $\lesssim 1$ intervals; since each connected component has length $\gtrsim 1$. It follows from \cite[Lemma~3.3]{christ} that for any $\varepsilon >0$, 
\begin{equation} \label{measurebound} m\left\{ \theta \in [0, \pi) :  \left\lvert t-\tau +2 \omega\left( \pi_{V_{\theta}}(z), z \right) - 2 \omega\left( \pi_{V_{\theta}}(\zeta), \zeta \right) \right\rvert < \varepsilon  \right\} \lesssim \frac{ \varepsilon^{1/2}}{ 2^{-k}}. \end{equation}
Hence 
\begin{multline*} 
\int_{\left\{ \theta \in E^{(2)} : \left\lvert t- \tau + 2\omega\left( \pi_{V_{\theta}}(z), z\right) - 2\omega\left(\pi_{V_{\theta}}(\zeta), \zeta \right) \right\rvert  < 2^{-{2j}} \right\}} \\
 \left\lvert \int_{\pi\left(A_j\right)} e^{2\pi i \left\langle \left(rie^{i \theta}, \rho\right), \left(z-\zeta, t-\tau +2 \omega\left( \pi_{V_{\theta}}(z), z \right) - 2 \omega\left( \pi_{V_{\theta}}(\zeta), \zeta \right)\right) \right\rangle}  \, d\rho \, dr \right\rvert \, d\theta \\ \lesssim \min\left\{2^{j+3k}, 2^{2j+k} \right\}. \end{multline*}
This can be justified as follows. For each $\theta$ in the domain of integration, the integral over the rectangle $\pi(A_j)$ can be written as a product of two integrals, and each integral can be calculated directly. The modulus of the $r$-integral is $\lesssim 2^{2k}$, and is also trivially $\lesssim 2^j$. The modulus of the $\rho$-integral is $\lesssim 2^{2j}$ trivially. Integrating these two bounds and using \eqref{measurebound} with $\varepsilon = 2^{-2j}$ gives the bound of $\min\left\{2^{j+3k}, 2^{2j+k} \right\}$. A similar argument gives that for any $l \leq 2j$,
\begin{multline*} \int_{\left\{ \theta \in E^{(2)} : 2^{-(l+1)} \leq \left\lvert t- \tau + 2\omega\left( \pi_{V_{\theta}}(z), z\right) - 2\omega\left(\pi_{V_{\theta}}(\zeta), \zeta \right) \right\rvert  < 2^{-l} \right\} } \\
 \left\lvert \int_{\pi\left(A_j\right)} e^{2\pi i \left\langle \left(rie^{i \theta}, \rho\right), \left(z-\zeta, t-\tau +2 \omega\left( \pi_{V_{\theta}}(z), z \right) - 2 \omega\left( \pi_{V_{\theta}}(\zeta), \zeta \right)\right) \right\rangle}  \, d\rho \, dr \right\rvert \, d\theta  \\ \lesssim \min\left\{2^{3k + l/2}, 2^{j +k +l/2} \right\}; \end{multline*}
the differences being that now the $\rho$-integral is $\lesssim 2^{l}$, and the bound from \eqref{measurebound} uses $\varepsilon = 2^{-l}$. By combining these two bounds with \eqref{pause} and summing the geometric series over $l$,
\begin{equation} \label{secondint} \eqref{pause} \lesssim \min\left\{2^{j+3k}, 2^{2j+k} \right\}. \end{equation}

It remains to bound the first integral in the right-hand side of \eqref{twoparts}. By the assumptions on the support of $\mu$ (from \eqref{alt1}--\eqref{alt2}), 
\[ \min\left\{ \left\lvert q -  e^{i \theta_0} \right\rvert, \left\lvert q +  e^{i \theta_0} \right\rvert \right\} < \epsilon^3, \]
and thus for any $\theta \in E$,
\begin{align*} \left\lvert \left\langle q, e^{i \theta} \right\rangle \right\rvert &\geq 1 - \left\lvert \left\langle q, i e^{i \theta} \right\rangle\right\rvert \\
&\geq 1 - \epsilon^3 -  \left\lvert \left\langle e^{i \theta_0}, i e^{i \theta} \right\rangle\right\rvert \\
&= 1- \epsilon^3 - \left\lvert \sin(\theta-\theta_0)\right\rvert \\
&\geq \epsilon^2/10, \end{align*}
by the definition of $E$ (see \eqref{Edefn}). The function $F$ from \eqref{Fdefn} therefore satisfies
\begin{align*} \left\lvert F'(\theta)\right\rvert &\geq 2\left\lvert z-\zeta\right\rvert \cdot \left\lvert z+\zeta\right\rvert\left( \left\lvert \left\langle p, e^{i \theta} \right\rangle  \cdot\right\rvert  \left\lvert \left\langle q, e^{i \theta} \right\rangle\right\rvert -  \left\lvert  \left\langle p, ie^{i \theta} \right\rangle\right\rvert  \cdot \left\lvert  \left\langle q, ie^{i \theta} \right\rangle \right\rvert\right) \\
&\geq 2\left\lvert z-\zeta\right\rvert \cdot \left\lvert z+\zeta\right\rvert\left( \epsilon^2/20  - \epsilon^3 \right) \\
&\gtrsim 2^{-2k}, \end{align*}
for any $\theta \in E^{(1)}$. By the mean value theorem, it follows that for any $\varepsilon >0$,
\[ m\left\{ \theta \in E^{(1)} : \left\lvert t-\tau +2 \omega\left( \pi_{V_{\theta}}(z), z \right) - 2 \omega\left( \pi_{V_{\theta}}(\zeta), \zeta \right) \right\rvert < \varepsilon  \right\} \lesssim \frac{\varepsilon}{ 2^{-2k}}. \]
 Summing over dyadic numbers $\varepsilon$ with $2^{-2j} \leq \varepsilon \lesssim 1$ yields that
\begin{equation} \label{firstint} \left\lvert \int_{A_j^{(1)}}  e^{2\pi i \left\langle \left(rie^{i \theta}, \rho\right), \left(z-\zeta, t-\tau +2 \omega\left( \pi_{V_{\theta}}(z), z \right) - 2 \omega\left( \pi_{V_{\theta}}(\zeta), \zeta \right)\right) \right\rangle}  \, d\rho \, d\theta \, dr \right\rvert \lesssim j2^{j+2k}. \end{equation} Combining \eqref{vertcase}, \eqref{twoparts}, \eqref{secondint}, and \eqref{firstint} gives \eqref{claimedbound} for any $k \in \{0, \dotsc, j\}$. This proves \eqref{claimedbound}, and as explained previously, this finishes the proof of the theorem.  \end{proof} 

\section{Remarks and further questions}
\label{remark}
\begin{sloppypar}
\begin{enumerate}
\item Theorem~\ref{mainthm} also holds if the Korányi metric on the domain is replaced by the (non-equivalent) parabolic metric $d((z,t), (\zeta, \tau)) = \left( \lvert z-\zeta\rvert^4 + \lvert t-\tau\rvert^2 \right)^{1/4}$, with only minor changes to the proof (around \eqref{vertcase}). It is not necessary to vary the metric on the codomain, since in $\mathbb{H}^1$ the two metrics are equivalent on any vertical subgroup.
\item Whether the inequality $\left\lvert \widehat{f_s}\right\rvert \lesssim f_{3-s}$ from Lemma~\ref{koranyipotential} holds for $s \in (0,1)$ is an open problem; the method of proof does not seem to extend to this range. The formula \eqref{fourierfunction} for $\widehat{f_s}$ is possibly just an artefact of the method. 
\item In \cite[Example~7.11 and Remark~7.12]{balogh}, examples are given of sets $A$ of any dimension $\dim A \in (1,2]$, for which the \emph{standard} energy method cannot yield any improvement over the lower bound of $(1+\dim A)/2$. However, the examples from~\cite{balogh} do not seem to preclude further improvement by a \emph{modified} energy method such as that used in the proof of Theorem~\ref{mainthm}.  \end{enumerate}\end{sloppypar}

\end{document}